\documentclass[a4paper,11pt]{amsart}
\usepackage[utf8]{inputenc}
\usepackage{graphicx}
\usepackage{latexsym}
\usepackage{amsmath}
\usepackage{amssymb,amsfonts}
\usepackage{amsthm}
\usepackage{hyperref}
\usepackage{multicol} 
\usepackage{multirow}
\usepackage{caption}
\usepackage{subcaption}
\usepackage{MnSymbol}
\usepackage{mathtools}
\usepackage{pgf,tikz}
\usepackage{array}
\tikzstyle{vertex}=[circle,draw, inner sep=0pt, minimum size=5pt] 
\newcommand{\vertex}{\node[vertex]}
\usetikzlibrary{decorations.markings}
\usepackage{lipsum}

\def \L {\mathfrak{L}}
\def \D {\mathfrak{D}}
\def \J {\mathbb{J}}

\def \I {\mathbb{I}}
\def \J {\mathbb{J}}
\newtheorem*{theorem*}{Theorem}
\newtheorem{theorem}{Theorem}

\newtheorem{cor}[theorem]{Corollary}

\begin{document}
\title[]{Spectral properties of edge Laplacian matrix}
\author{Shivani Chauhan}
\address{Department of Mathematics, Shiv Nadar Institution of Eminence, Dadri, U.P, 201314}
\email{sc739@snu.edu.in}

\author{A. Satyanarayana Reddy}
\address{Department of Mathematics, Shiv Nadar Institution of Eminence, Dadri, U.P, 201314}
\email{satya.a@snu.edu.in}

\begin{abstract}
Let $N(X)$ be the Laplacian matrix of a directed graph obtained from the edge adjacency matrix of a graph $X.$ In this work, we study the bipartiteness property of the graph with the help of $N(X).$ We computed the spectrum of the edge Laplacian matrix for the regular graphs, the complete bipartite graphs, and the trees. Further, it is proved that given a graph $X,$ the characteristic polynomial of $N(X)$ divides the characteristic polynomial of $N(X^{\prime\prime}),$ where $X^{\prime\prime}$ denote the Kronecker double cover of $X.$

\vspace{2mm}
\vspace{2mm}
\end{abstract}

\maketitle
\section{Introduction}

The {\em Laplacian matrix} $\L(X)$ or $\L$ of a directed graph $X$ is defined as $\L(X)=\D(X)-A(X)$, where $A(X)$ is the adjacency matrix of $X$ and $\D(X)$ is the diagonal matrix of the row sums of $A(X).$ In \cite{wu2005algebraic,wu2005rayleigh}, Chai Wah Wu defined the algebraic connectivity of directed graphs. It is useful in deriving synchronization criteria for arrays of coupled dynamical systems for both constant and time-varying coupling. In this work, we are studying the Laplacian matrix of a special class of directed graphs, for that we first define the edge adjacency matrix of a graph.

There are other definitions of Laplacian matrices of directed graphs (see \cite{bapat1999generalized,chung2005laplacians}). Let $X=(V(X),E(X))$ be a graph with $|V(X)|=n, \;|E(X)|=m.$ We orient the edges of $X$ arbitrarily, label them as $e_1,e_2,\ldots,e_m$ and also $e_{m+i}=e_i^{-1},\;1\le i\le m,$ where $e_k^{-1}$ denotes the edge $e_k$
with the direction reversed. Then the {\em edge adjacency matrix} of $X$, denoted by $M(X)$ or $M$ if $X$ is clear from the context, is defined as 
 $$M_{ij}=\begin{cases}
          1 & \mbox{if $t(e_i)=s(e_j)$\; and\; $s(e_i)\ne t(e_j)$,}\\
          0 & \mbox{otherwise.}
                   \end{cases}$$
where $s(e_i)$ and $t(e_i)$ denote the starting and the terminal vertex of $e_i,$ respectively. For further information about the matrix $M$, one can refer \cite{chauhan2023double,horton2006ihara,terras2010zeta}.

The {\em edge Laplacian matrix} $N(X)$ or $N$ of a graph $X$ is defined as $N=D-M$, where $D$ is the diagonal matrix of the row sums of $M$ {\it i.e.,} the $i^{th}$ diagonal entry equals to $deg(t(e_i))-1$, and $deg(v)$ denotes the degree of the vertex $v.$  Recall, a graph is {\em strongly connected} if for every ordered pair of vertices $(v, w)$, there exists a directed path from $v$ to $w.$
In Theorem 11.10 of \cite{terras2010zeta}, it is proved that if $X$ is a connected graph that is not a cycle graph and does not contain a pendant vertex, then matrix $M$ is irreducible. Therefore the directed graph obtained from the matrix $M$ is strongly connected, hence 0 is a simple eigenvalue of $N$. The process of computation of matrix $N(C_3)$ is given in Figure \ref{fig:ex N}.
\par
 In the next section, we describe the structure of the matrix $N$ which is useful to determine the bipartiteness property of the graph. Further, the spectrum of various families of graphs is provided, in particular for regular graphs, trees, and complete bipartite graphs. Lastly, it is proved that $\phi_{N(X)}$ divides $\phi_{N(X^{\prime\prime})},$ where $\phi_A$ and $X^{\prime\prime}$ denote the characterstic polynomial of a matrix $A$ and the Kronecker product of a graph $X$ with $K_2$, respectively and $K_n$ denotes the complete graph on $n$ vertices. Recall, the {\em Kronecker product} $X_ 1\times X_2$ of graphs $X_1$ and $X_2$ is a graph such that the vertex set is $V(X_1)\times V(X_2)$, vertices $(x_1,x_2)$ and $(x_1^\prime,x_2^\prime)$ are adjacent if $x_1$ is adjacent to $x_1^\prime$ and $x_2$ is adjacent to $x_2^\prime.$ In particular given a graph $X,$ $X \times K_2$ is called the {\em Kronecker double cover} of $X.$
 From now onwards, the edge adjacency matrix of a graph is denoted by $M$ and the diagonal matrix of its rows sums by $D.$ 

\begin{figure}[h!]
    \begin{subfigure}[ht!]{0.3\textwidth}
\begin{tikzpicture}
\vertex (1) at (0,0){};
\vertex (2) at (2,0){};
\vertex (3) at (1,1){};
     \path[-]
    (1) edge (2)
    (2) edge (3)
    (1) edge (3)
    ; 
\end{tikzpicture}
\caption{$X$}
\label{fig:C_3}
\end{subfigure}
\hfill
\begin{subfigure}[ht!]{0.3\textwidth}
\tikzset{every picture/.style={line width=0.75pt}} %set default line width to 0.75pt        
\begin{tikzpicture}[x=0.75pt,y=0.75pt,yscale=-0.5,xscale=0.5]
\draw   (100,115) .. controls (100,112.24) and (102.24,110) .. (105,110) .. controls (107.76,110) and (110,112.24) .. (110,115) .. controls (110,117.76) and (107.76,120) .. (105,120) .. controls (102.24,120) and (100,117.76) .. (100,115) -- cycle ;
%Straight Lines [id:da30650579205209105] 
\draw    (100.48,120.21) -- (41.42,199.18) ;
%Straight Lines [id:da5156374467879599] 
\draw    (110,119) -- (112.48,122.84) -- (172,203.5) ;
%Straight Lines [id:da9244115533394988] 
\draw    (44.61,202.3) -- (167.5,208) ;

%Shape: Ellipse [id:dp011124939741658624] 
\draw   (34,202.3) .. controls (34,205.17) and (36.37,207.5) .. (39.3,207.5) .. controls (42.23,207.5) and (44.61,205.17) .. (44.61,202.3) .. controls (44.61,199.43) and (42.23,197.1) .. (39.3,197.1) .. controls (36.37,197.1) and (34,199.43) .. (34,202.3) -- cycle ;
%Shape: Circle [id:dp5942141070151314] 
\draw   (167.5,208) .. controls (167.5,210.49) and (169.51,212.5) .. (172,212.5) .. controls (174.49,212.5) and (176.5,210.49) .. (176.5,208) .. controls (176.5,205.51) and (174.49,203.5) .. (172,203.5) .. controls (169.51,203.5) and (167.5,205.51) .. (167.5,208) -- cycle ;
\draw   (126,129.5) -- (124.67,138.67) -- (115.88,135.76) ;
\draw   (151,165.5) -- (141.18,160.84) -- (140.52,171.69) ;
\draw   (63,195.5) -- (75,202.5) -- (63,209.5) ;
\draw   (128.84,211.14) -- (118.63,205.82) -- (128.49,199.88) ;
\draw   (95,142.5) -- (83.8,143.58) -- (83.41,132.33) ;
\draw   (52.61,173.37) -- (62.62,170.31) -- (65,180.5) ;

% Text Node
\draw (128,114) node [anchor=north west][inner sep=0.75pt]   [align=left] {$e_1$};
% Text Node
\draw (154,151) node [anchor=north west][inner sep=0.75pt]   [align=left] {$e_1^{-1}$};
% Text Node
\draw (116,179) node [anchor=north west][inner sep=0.75pt]   [align=left] {$e_2$};
% Text Node
\draw (59,210) node [anchor=north west][inner sep=0.75pt]   [align=left] {$e_2^{-1}$};
% Text Node
\draw (30,145) node [anchor=north west][inner sep=0.75pt]   [align=left] {$e_3$};
% Text Node
\draw (61,92) node [anchor=north west][inner sep=0.75pt]   [align=left] {$e_3^{-1}$};
\end{tikzpicture}
\caption{Oriented graph of $X$}
\label{fig:OC_3}
\end{subfigure}
\hfill
\begin{subfigure}[h!]{0.3\textwidth}
\begin{tikzpicture}
\vertex (1) at (0,0)[label=left:$e_1$]{};
\vertex (2) at (2,0)[label=right:$e_2$]{};
\vertex (3) at (1,1)[label=above:$e_3$]{};
\vertex (4) at (0,-2)[label=left:$e_1^{-1}$]{};
\vertex (5) at (2,-2)[label=right:$e_2^{-1}$]{};
\vertex (6) at (1,-1)[label=above:$e_3^{-1}$]{};
     \path[->]
    (1) edge (2)
    (2) edge (3)
    (3) edge (1)
    (5) edge (4)
    (6) edge (5)
    (4) edge (6)
    ; 
\end{tikzpicture}
\caption{Directed graph obtained from $M(X)$}
\label{fig:DM}
\end{subfigure}
\hfill
\begin{subfigure}[h!]{0.3\textwidth}
\begin{tabular}{|c|ccc|ccc|}
\hline
&$e_1$&$e_2$&$e_3$&$e_1^{-1}$&$e_2^{-1}$&$e_3^{-1}$\\
\hline
$e_1$&1 & -1 & 0 & 0 & 0 & 0\\
$e_2$&0 & 1 & -1 & 0 & 0 & 0\\
$e_3$&-1 & 0 & 1 & 0 & 0 & 0\\
\hline
$e_1^{-1}$&0 & 0 & 0 & 1 & 0 & -1\\
$e_2^{-1}$&0 & 0 & 0 & -1 & 1 & 0\\
$e_3^{-1}$&0 & 0 & 0 & 0 & -1 & 1\\  
\hline
\end{tabular}
\caption{$N(X)$}
\label{fig:N(X)}
\end{subfigure}
 \caption{}
    \label{fig:ex N}
\end{figure}
\section{Spectrum of edge Laplacian matrix}
The collection of all the eigenvalues of a matrix together with its multiplicities is known as the {\em spectrum} of that matrix. If $\lambda_1>\lambda_2>\ldots>\lambda_k$ are the distinct eigenvalues of a matrix $A,$ then the spectrum of $A$ is defined by 
$$\begin{Bmatrix}
\lambda_1 & \lambda_2 & \cdots & \lambda_k\\
m_1 & m_2 & \cdots & m_k
\end{Bmatrix},$$ where $m_i$ is the algebraic multiplicity of eigenvalue $\lambda_i$ for $1\leq i\leq k.$ The matrix $N$ has the following structure \begin{equation}\label{eqn :1}
N=D-M=\left[\begin{array}{c|c}
  P & Q \\
  \hline
  R & S
\end{array}
\right],
\end{equation}
where $P,Q,R,S$ are $m\times m$ matrices. From the structure of matrix $M,$ we have that $Q$ and $R$ are symmetric matrices. If $X$ is regular, then $P=S^T$ and $N^T=JNJ,$ where $J=\begin{pmatrix}
0 & \I_{m}\\
\I_{m} & 0
\end{pmatrix}$ and $\I_m$ denotes the identity matrix of order $m$.
The row and column sums of the matrix $N$ are zero for regular graphs which implies that the cofactors of any two elements of $N$ are equal \cite[Lemma 4.2]{bapat2010graphs}. It is interesting to note that the sum of the blocks of $N,$ $P+Q+R+S$ is the Laplacian matrix of $L(X)$, where $L(X)$ denotes the line graph of $X$.
As
    $$N=D-M=\begin{bmatrix}
D_1 & 0\\
0 & D_2
\end{bmatrix}- \begin{bmatrix}
A & B\\
C & D
\end{bmatrix},$$
the sum of the blocks of $M$ is equal to $A(L(X))$ \cite[p.7]{horton2006ihara}
and $(D_1+D_2)_{ii}=deg(t(e_i))+deg(t(e_i^{-1}))-2,$ where $1\leq i\leq m$.

It is well known from \cite{grone1990laplacian} that a graph $X$ is bipartite if and only if the Laplacian matrix and the signless Laplacian matrix of $X$ are similar. Our next result is analogous to this. 

\begin{theorem}\label{thm:similar}
A graph $X$ is bipartite if and only if $D+M$ and $D-M$ are similar, where $M$ is irreducible.
\end{theorem}

\begin{proof}
Suppose that $X$ is bipartite with $m$ edges. Let $\{v_1,v_2,\ldots,v_m\}$ and $\{v_1^\prime,v_2^\prime,\ldots,v_n^\prime\}$ be the vertex bipartition of $X.$ Choose an orientation in $X$ in such a way that $e_i's$ are the directed edges from $v_i$ to $v_j^\prime$ for all $1\leq i\leq m, 1\leq j \leq n$, then $M$ can be expressed in the following form
 \begin{equation}\label{eqn:bip}
 M=\left[\begin{array}{c|c}
  0 & B\\
  \hline
  C & 0
\end{array}
\right].
\end{equation} It is simple to check that $P^T(D+M)P=D-M$, where $P=\begin{bmatrix}
\I_m & 0\\
0 & -\I_m
\end{bmatrix}.$ 

Conversely, suppose that $D+M$ and $D-M$ are similar, where $M$ is irreducible. Clearly, $0$ is an eigenvalue of $D+M$. Let $x=(x_1,x_2,\ldots,x_{2m})$ be the eigenvector corresponding to the eigenvalue $0$ of $D+M$. Choose $x_i$ to be the maximum absolute entry of $x.$ As $((D+M)x)_i=0,$ we have $\sum_{j=1,i\neq j}^{2m}m_{ij}x_j+d_{ii}x_i=0$. We obtain $d_{ii}x_i=-\sum x_j,$ where the summation is over the edges to which $e_i$ is feeding and these edges are $d_{ii}$ in number. By the maximality of $x_i,$ we have $x_i=-x_j$ for all edges into which $e_i$ is feeding.
Now, we explain the construction of the vertex partition. We put the vertex $e_i$ in one partition, say $V_1$, and the vertices $e_j's$ into which $e_i$ is feeding in the partition $V_2$ and the edges are directed from $e_i$ to $e_j's.$ Again we choose the maximum absolute entry of $x$ and continue the same process. This shows $M$ has the structure given in (\ref{eqn:bip}), hence $X$ is bipartite.
\end{proof}

We shall now present the spectrum for several graph families. The spectrum of a regular graph can be easily obtained from Theorem 2.2 in \cite{glover2021some}. Let $X$ be a $k$-regular graph on $n$ vertices, then the eigenvalues of $N$ are
$$
k-1\pm 1,k-1-\left(\frac{\lambda_i \pm \sqrt{\lambda_i^2-4(k-1)}}{2}\right),(i=1,\ldots,n)$$
where $\lambda_i \in spec(A(X))$ and $k-1\pm 1$ each have multiplicity $m-n.$
Later, we computed the eigenvalues of $N$ for trees and complete bipartite graphs in Theorems \ref{thm:inteigen} and \ref{thm:cbi}, respectively. In Table \ref{tab:6vertices}, the spectrum of $N(X)$ is given, where $X$ is a tree on $6$ vertices. We state Theorem \ref{thm:nil} which is required in the proof of Theorem \ref{thm:inteigen}.

\begin{table}[h!]
    \centering
    \scalebox{0.7}{
    \begin{tabular}{|c|c|}
    \hline
    $X$ & Spectrum of $N(X)$\\
    \hline
   \begin{tikzpicture}
    \vertex (1) at (0,0) {};
    \vertex (2) at (1,0) {};
    \vertex (3) at (2,0) {};
    \vertex (4) at (3,0) {};
    \vertex (5) at (4,0) {};
    \vertex (6) at (5,0) {};
    \path[-]
    (1) edge (2)
    (2) edge (3)
    (3) edge (4)
    (4) edge (5)
    (6) edge (5)
    ;
   \end{tikzpicture} &
 $\begin{Bmatrix}
0 & 1 \\
2 & 8
\end{Bmatrix}$  
   \\
    \hline
       \begin{tikzpicture}
\vertex (6) at (-1,0) {};
    \vertex (1) at (0,0) {};
    \vertex (2) at (1,0) {};
    \vertex (3) at (2,0) {};
    \vertex (4) at (3,0.5) {};
    \vertex (5) at (3,-0.5) {};
    ;
    \path[-]
    (1) edge (2)
    (2) edge (3)
    (3) edge (4)
    (3) edge (5)
    (6) edge (1)
    ;
\end{tikzpicture}  & $\begin{Bmatrix}
0 & 1 & 2\\
3 & 4 & 3
\end{Bmatrix}$ \\
\hline
\begin{tikzpicture}[scale=0.7]  
    \vertex (1) at (1,1) {};
    \vertex (2) at (2,1) {};
    \vertex (3) at (3,2) {};
    \vertex (4) at (3,0) {};
    \vertex (5) at (4,2) {};
    \vertex (6) at (4,0) {};
    \path[-]
    (1) edge (2)
    (2) edge (3)
    (2) edge (4)
    (3) edge (5)
    (4) edge (6)
     ;
     \end{tikzpicture}&
    $\begin{Bmatrix}
0 & 1 & 2\\
3 & 4 & 3
\end{Bmatrix}$
     \\
     \hline
\begin{tikzpicture}
    \vertex (1) at (0,0) {};
    \vertex (2) at (1,0) {};
    \vertex (3) at (2,0) {};
    \vertex (4) at (3,0.5) {};
    \vertex (5) at (3,-0.5){};
    \vertex (6) at (3,0) {};
    ;
    \path[-]
    (1) edge (2)
    (2) edge (3)
    (3) edge (4)
    (3) edge (5)
    (3) edge (6)
    ;
\end{tikzpicture} & 
$\begin{Bmatrix}
0 & 1 & 3\\
4 & 2 & 4
\end{Bmatrix}$\\
\hline
\begin{tikzpicture}[scale=0.7]
    \vertex (1) at (1,1) {};
    \vertex (2) at (3,1) {};
    \vertex (3) at (4,2) {};
    \vertex (4) at (4,0) {};
    \vertex (5) at (0,2) {};
    \vertex (6) at (0,0) {};
    
    \path[-]
    (1) edge (5)
    (1) edge (6)
    (1) edge (2)
    (2) edge (3)
    (4) edge (2)
     ;
\end{tikzpicture} & $\begin{Bmatrix}
0 & 2 \\
4 & 6 
\end{Bmatrix}$
\\
\hline
\begin{tikzpicture}
    \vertex (1) at (2,0.5) {};
    \vertex (2) at (1,0) {};
    \vertex (3) at (2,0) {};
    \vertex (4) at (3,0.5){};
    \vertex (5) at (3,-0.5){};
    \vertex (6) at (3,0) {};
    
    \path[-]
    (1) edge (3)
    (2) edge (3)
    (3) edge (4)
    (3) edge (5)
    (3) edge (6)
    ;
\end{tikzpicture} &
$\begin{Bmatrix}
0 & 4 \\
5 & 5 
\end{Bmatrix}$
\\
\hline
\end{tabular}}
    \caption{}
    \label{tab:6vertices}
\end{table}

\begin{theorem} \cite[Theorem 3.2]{torres2020non}
\label{thm:nil}
Let $M$ be the edge adjacency matrix of a graph $X.$ Then $M$ is a nilpotent matrix if and only if $X$ is a forest.
\end{theorem}
  
\begin{theorem}\label{thm:inteigen}
Let $X$ be a connected graph. Then the eigenvalues of $N$ are the same as the eigenvalues of $D$ if and only if $X$ is a tree.
\end{theorem}
  
  \begin{proof}
  Suppose that $X$ is a tree on $m$ edges. Let $\phi_D(\lambda)=\lambda^{2m}+d_1\lambda^{2m-1}+d_2\lambda^{2m-2}+\cdots+d_{2m-1}\lambda+d_{2m}$ and $
        \phi_N(\lambda)=\lambda^{2m}+n_1\lambda^{2m-1}+n_2\lambda^{2m-2}+\cdots+n_{2m-1}\lambda+n_{2m}$
    be the characterstic polynomial of matrix $D$ and $N,$ respectively. We want to show that $d_i=n_i$, where $1\leq i \leq 2m.$ In order to prove the claim, we use the fact that the sum of the products of the eigenvalues of a matrix $A$ taken $k$ at a time equals the sum of the $k\times k$ principal minors of $A$ \cite[p.53]{horn2012matrix}.
    It is simple to observe that $d_1=n_1.$ The $2\times 2$ principal submatrix of $N$ has the following structure,
    $$B_2=\begin{pmatrix}
    d_{ii} & -m_{ij}\\
    -m_{ji} & d_{jj}
    \end{pmatrix},$$ where $m_{ij}$ denotes the $ij^{th}$ entry of $M$.
    By the definition of $M$, $det(B_2)=d_{ii}d_{jj}$, therefore $d_2=n_2$. 
    Let $B_k$ denote the $k\times k$ principal submatrix of $N$, whose rows and columns are  indexed by $i_1>i_2>\ldots>i_k.$ It is well known that \begin{equation}\label{eqn:det}
    det(A)=\sum_{\sigma\in S_n}sgn(\sigma)\prod_{i=1}^{n}a_{i\sigma(i)},
    \end{equation}
    where the summation is over all permutation $\{1,2,\ldots,n\}$.
    As $X$ is a tree, the directed graph obtained from the matrix $M$ has no cycles. By (\ref{eqn:det}), $det(B_k)=d_{i_1i_1}d_{i_2i_2}\ldots d_{i_ki_k}$, which implies $d_k=n_k$.
    \par
    Conversely, suppose that $\phi_D(\lambda)=\phi_N(\lambda)$. Let $\phi_M(\lambda)=\lambda^{2m}+m_1\lambda^{2m-1}+m_2\lambda^{2m-2}+\ldots+m_{2m-1}\lambda+m_{2m}$. We want to show that $m_i=0\; \forall \;1\leq i \leq 2m.$ The result is proved using strong induction. Clearly, $m_1=m_2=0$. Assume that $m_i=0 \forall 1\leq i\leq k$. As $m_{k+1}=$ sum of $(k+1) \times (k+1)$ principal minors of $M$, by (\ref{eqn:det})
\begin{equation}\label{eqn:det1}
m_{k+1}=\sum_{i_1<i_2<\ldots<i_{k+1}}\sum_{\sigma \in S_{k+1}}sgn(\sigma)\prod_{\ell=1}^{k+1}m_{i_{\ell}\sigma(i_{\ell})},    
\end{equation}
 where $1\leq i_1,i_2,\ldots,i_{k+1}\leq 2m.$ In the inner summation of (\ref{eqn:det1}), all the entries corresponding to permutations which are not a cycle of length $k+1$ are zero, by induction hypothesis. Hence, we are left with the following expression
$$m_{k+1}=\sum_{i_1<i_2<\ldots<i_{k+1}}\sum_{\sigma \in C\cup \{e\}}(-1)^{k+1}m_{i_1\sigma(i_1)}m_{i_2\sigma(i_2)}\ldots m_{i_{k+1}\sigma(i_{k+1})},$$ where $C$ denotes the cycle of length $k+1$. As sum of $k\times k$ principal minors of $N$ is equal to the sum of $k\times k$ principal minors of $D$, we deduce the following expression
\begin{equation*}
\sum_{i_1<i_2<\ldots<i_{k+1}}\left(\prod_{j=1}^{k+1}n_{i_ji_j}+\sum_{\sigma\in S_{k+1},\sigma \neq \{e\}}sgn(\sigma)\prod_{\ell=1}^{k+1}n_{i_\ell\sigma(i_\ell)}\right)= \sum_{i_1<i_2<\ldots<i_{k+1}}\prod_{j=1}^{k+1}n_{i_ji_j}.    
\end{equation*}

By above Equation we obtain \begin{equation}\label{eqn:nil:2}
\sum_{i_1<i_2<\ldots<i_{k+1}}\sum_{\sigma \in S_{k+1},\sigma\neq \{e\}}sgn(\sigma)\prod_{\ell=1}^{k+1}n_{i_\ell\sigma(i_\ell)}=0.
\end{equation}
As $N=D-M$,
$n_{i_\ell \sigma(i_\ell)}=\begin{cases}
d_{i_\ell i_\ell} & \mbox{if $\sigma(i_\ell)=i_\ell$}\\
-m_{i_\ell\sigma(i_\ell)} & \mbox{if $\sigma(i_\ell)\neq i_\ell$}
\end{cases}$. In the inner summation of (\ref{eqn:nil:2}), all the terms in the permutations except those terms corresponding to full cycle permutation are zero, by induction hypothesis. Equation (\ref{eqn:nil:2}) reduces to
$$\sum_{i_1<i_2<\cdots<i_{k+1}}\sum_{\sigma \in C ,\sigma\neq \{e\}}(-1)^{k+1}(-1)^{k+1}m_{i_1\sigma(i_1)}m_{i_2\sigma(i_2)}\cdots m_{i_{k+1}\sigma(i_{k+1})}=0.$$
Hence $m_{k+1}=0$ and
by Theorem \ref{thm:nil}, the proof is complete.
\end{proof}

\begin{cor}
Let $X$ be a tree on the $n$ vertices and $(d_1,d_2,\ldots,d_n)$ be the degree sequence of $X$. Then
$d_i-1$ is an eigenvalue of $N$ with multiplicity $d_i$, where $1\leq i\leq n$.
\end{cor}
We now define the Kronecker product of the matrices and the Schur complement as it is required to prove our next result.
The Kronecker product of the matrices $A=[a_{ij}]$ and $B$ is defined as the partitioned matrix $[a_{ij}B]$ and is denoted by $A\otimes B.$ Let $H$ be an $n\times n$ matrix partitioned as $H=\begin{bmatrix}
A_{11} & A_{12}\\
A_{21} & A_{22}
\end{bmatrix},$ where $A_{11}$ and $A_{22}$ are square matrices. If $A_{11}$ is non-singular, then the Schur complement of $A_{11}$ in $H$ is defined to be the matrix $A_{22}-A_{21}A_{11}^{-1}A_{12}.$ For Schur complement, we have
$det(H)=det(A_{11})det(A_{22}-A_{21}A_{11}^{-1}A_{12})$ and if $A_{11}A_{12}=A_{12}A_{11}$, then $det(H)=det(A_{22}A_{11}-A_{21}A_{12})$.
Similarly if $A_{22}$ is non-singular, then the Schur complement of $A_{22}$ in $H$ is defined to be the matrix $A_{11}-A_{12}A_{22}^{-1}A_{21}$ and we can obtain $det(H)=det(A_{22})det(A_{11}-A_{12}A_{22}^{-1}A_{21})$. If $A_{21}A_{22}=A_{22}A_{21}$, then $det(H)=det(A_{11}A_{22}-A_{12}A_{21})$. 

\begin{theorem}\label{thm:cbi}
Let $X=K_{p,q}$ be the complete bipartite graph on the $n$ vertices. Then the spectrum of $N(K_{p,q})$ is equal to  $$\begin{Bmatrix}
   0 & u & \frac{u\pm \sqrt{v+4}}{2} & \frac{u\pm \sqrt{v+4(1-q)}}{2} & \frac{u\pm \sqrt{v+4(1-p)}}{2}\\
  1 &1 & (q-1)(p-1) & p-1 & q-1
\end{Bmatrix},$$ where $u=p+q-2$ and $v=(p-q)^2$.
\end{theorem}

\begin{proof}
Let $V(K_{p,q})=\{v_1,v_2,\ldots,v_p\}\cup \{v_1^\prime,v_2^\prime,\ldots,v_q^\prime\}$ be a bipartition of $X$. Choosing an orientation that is defined in the proof of Theorem \ref{thm:similar}, then $N(K_{p,q})$ can be written in the following form
$$N(K_{p,q})=\begin{bmatrix}
(p-1)(U) &  \J_p\otimes(-\I_q)+U\\
V & (q-1)(U)
\end{bmatrix},$$ where $U=\I_p \otimes \I_q,V=\I_p\otimes(-A(K_q)) ,$ and $\J_n$ denotes the matrix of order $n$ with all entries one.
The characterstic polynomial of $N(X)$ is $$\phi_{N}(x)=\begin{vmatrix}
(p-1-x)(U) &  \J_p\otimes(-\I_q)+U\\
V & (q-1-x)(U)
\end{vmatrix}=0.$$
A simple check shows that $$V((q-1-x)(U))=((q-1-x)(U))V=\I_p\otimes((x-q+1)A(K_q)).$$
By the Schur complement formula, we obtain the following Equation
\begin{equation*}
\begin{split}
\phi_N(x)&=det(((p-1-x)U)((q-1-x)U)-(\J_p\otimes (-\I_q) +U)V)\\
&=det((p-1-x)(q-1-x)(U)-\J_p\otimes A(K_q)-V).\\
\end{split}
\end{equation*}
For the sake of convenience, assume that $$S^\prime=(p-1-x)(q-1-x)(U)-\J_p\otimes A(K_q)-V,$$ $$T^\prime=(1-p)A(K_q)+(p-x-1)(q-x-1)\I_q,U^\prime=A(K_q)+(p-x-1)(q-x-1)\I_q.$$ Let $\lambda$ be an eigenvalue of $T^\prime$ with its corresponding eigenvector $v$, then $S^\prime V=\lambda V,$ where $V$ is a $p$ dimensional vector $(v,v,\ldots,v)^T$.\\
Let $\lambda^\prime$ be an eigenvalue of $U^\prime$ with corresponding eigenvector $v^\prime$. By a straightforward calculation, one can see that the following linearly independent vectors $$\{(v^\prime,-v^\prime,0,0,\ldots,0)^T,
(v^\prime,0,-v^\prime,0,\ldots,0)^T,\ldots,(v^\prime, 0,0,0,\ldots,-v^\prime)^T
\}$$ are the eigenvectors of $S^\prime$. As the sum of the multiplicities of the obtained eigenvalues equals the number of vertices in the directed graph obtained from $M$. Therefore,
$\phi_N(x)=(det(U^\prime))^{p-1}det(T^\prime).$
As $A(K_q)=\J_q-\I_q$ and the eigenvalues of $\J_q$ are $0$ and $q$ with multiplicity $q-1$ and $1$, respectively. From here the eigenvalues of $N$ can be deduced.
\end{proof}
The following result is about the eigenvalues of symmetric block matrices that will be used to prove the next result.

\begin{theorem} \cite{davis2013circulant}\label{thm:unionspectra}
 Let $H=\begin{bmatrix}
A^\prime & B^\prime\\
B^\prime & A^\prime
\end{bmatrix}$ be a symmetric $2\times 2$ block matrix, where $A^\prime$ and $B^\prime$ are square matrices of same order. Then the spectrum of $H$ is the union of the spectra of $A^\prime+B^\prime$
and $A^\prime-B^\prime$.
\end{theorem}

\begin{theorem}\label{thm:coverN}
Let $X$ be a connected graph, then $\phi_{N(X)}$ divides $\phi_{N(X^{\prime\prime})}$.
\end{theorem}

\begin{proof}
We know that $N(X)=D(X)-M(X)=\begin{pmatrix}
D_1-A & -B\\
-C & D_2-D
\end{pmatrix}$. From the proof of Theorem 4.2 in Horton~\cite{horton2006ihara}, $M(X^{\prime\prime})=\begin{pmatrix}
0 & MJ\\
JM & 0
\end{pmatrix},$ therefore $$N(X^{\prime\prime})=D(X^{\prime\prime})-M(X^{\prime\prime})=\begin{pmatrix}
D_1 & 0 & -B & -A\\
0 & D_2 & -D & -C\\
-C & -D & D_2 & 0\\
-A & -B & 0 & D_1
\end{pmatrix}
.$$ Note that $P^TN(X^{\prime\prime})P=\begin{pmatrix}
D_1 & 0 & -A & -B\\
0 & D_2 & -C & -D\\
-A & -B & D_1 & 0\\
-C & -D & 0 & D_2
\end{pmatrix},$ where $P=\begin{pmatrix}
I_m & 0 & 0 & 0\\
0 & I_m & 0 & 0\\
0 & 0 & 0 & I_m\\
0 & 0 & I_m & 0
\end{pmatrix}.$
The result follows from Theorem \ref{thm:unionspectra}.
\end{proof}

\section{Acknowledgement}
The authors would like to thank the handling editor and the anonymous reviewers for their careful reading of the manuscript.


\begin{thebibliography}{9}
\bibitem{bapat2010graphs} R.B. Bapat, {\em Graphs and matrices}, volume 27, Springer, (2010).
\bibitem{bapat1999generalized} R.B. Bapat, J. W. Grossman and D. M. Kulkarni, {\em Generalized matrix tree theorem for mixed graphs}, Linear and Multilinear Algebra, 46(40): 299-312, (1999).
\bibitem{chauhan2023double} S. Chauhan and A.S. Reddy, {\em On the double covers of a line graph}, arXiv preprint  arXiv:2202.04756v2, (2023).
\bibitem{chung2005laplacians} F. Chung, {\em Laplacians and the Cheeger inequality for directed graphs}, Annals of Combinatorics, 9(1):1-19, (2005).
\bibitem{davis2013circulant} P.J. Davis, {\em Circulant matrices}, American Mathematical Soc., (2013).
\bibitem{glover2021some} C. Glover and M. Kempton, {\em Some spectral properties of the non-backtracking matrix of a graph}, Linear Algebra and its Applications, 618:37-57, (2021).
\bibitem{grone1990laplacian}R. Grone, R. Merris and V. Sunder, {\em The Laplacian spectrum of a graph}, SIAM Journal on matrix analysis and applications, 11(2):218-238, (1990).
\bibitem{horn2012matrix} R.A. Horn and C. R. Johnson, {\em Matrix analysis}, Cambridge university press, (2012).
\bibitem{horton2006ihara} M.D Horton, {\em Ihara zeta functions of irregular graphs}, University of California, San Diego, (2006).
\bibitem{terras2010zeta} A. Terras, {\em Zeta functions of graphs: a stroll through the garden}, volume 128, Cambridge University Press, (2010).
\bibitem{torres2020non} L. Torres, {\em Non-backtracking spectrum: Unitary eigenvalues and diagonalizability}, arXiv preprint arXiv:2007.13611, (2020).
\bibitem{wu2005algebraic} C.W. Wu, {\em Algebraic connectivity of directed graphs}, Linear and multilinear algebra, 53(3):203-223, (2005).
\bibitem{wu2005rayleigh} C.W. Wu, {\em On Rayleigh–Ritz ratios of a generalized Laplacian matrix of directed graphs}, Linear algebra and its applications, 402:207-227, (2005).
    
\end{thebibliography}
\end{document}